\documentclass[11pt]{amsart}
\usepackage{amssymb, latexsym, mathrsfs, color}
\usepackage[colorlinks=true, pdfstartview=FitV,linkcolor=blue,citecolor=black,urlcolor=black]{hyperref}
\usepackage{fullpage}
\usepackage[all]{xy}

\newtheorem {theorem}    {Theorem}[section]

\newtheorem {lemma}      [theorem]    {Lemma}
\newtheorem {corollary}  [theorem]    {Corollary}
\newtheorem {proposition}[theorem]    {Proposition}

\newtheorem {property} [theorem] {Property}

\theoremstyle{definition}
\newtheorem{definition}[theorem]{Definition}
\newtheorem{remark}[theorem]{Remark}
\newtheorem{ex}[theorem]{Example}

\def\la{\langle}
\def\ra{\rangle}

\def\Z{\mathbb{Z}}
\def\R{\mathbb{R}}
\def\C{\mathbb{C}}

\newcommand{\defeq}{\mathbin{:=}}

\def\a{\alpha}

\numberwithin{equation}{section}

\newenvironment{red}{\relax\color{red}}{\relax}
\newenvironment{blue}{\relax\color{blue}}{\hspace*{.5ex}\relax}

\newcommand{\ber}{\begin{red}}
\newcommand{\er}{\end{red}}
\newcommand{\beb}{\begin{blue}}
\newcommand{\eb}{\end{blue}}

\begin{document}

\title[Entirety of cuspidal Eisenstein series on  Kac--Moody groups]{Entirety of certain cuspidal Eisenstein series on Kac--Moody groups}

\date{\today}

\author[L. Carbone]{Lisa Carbone$^{\dagger}$}
\address{Department of Mathematics, Rutgers University, Piscataway, NJ 08854-8019, U.S.A.}
\email{carbonel@math.rutgers.edu}
\thanks{$^{\dagger}$This work was
supported in part by the Simons Foundation, Mathematics and Physical Sciences Collaboration Grants for Mathematicians, 
Award Number: 422182.}

\author[K.-H. Lee]{Kyu-Hwan Lee$^{\star}$}
\thanks{$^{\star}$This work was partially supported by a grant from the Simons Foundation (\#712100).}

\address{Department of
Mathematics, University of Connecticut, Storrs, CT 06269, U.S.A.}
\email{khlee@math.uconn.edu}

\author[D. Liu]{Dongwen Liu$^{\diamond}$}
\thanks{$^{\diamond}$This work was partially supported by  the Fundamental Research Funds for the Central Universities 2016QNA2002.}
\address{School of Mathematical Sciences, Zhejiang University, Hangzhou 310027, P.R. China}\email{maliu@zju.edu.cn}


\begin{abstract}  Let $G$ be an infinite-dimensional representation-theoretic Kac--Moody  group associated to a nonsingular symmetrizable generalized Cartan matrix. We consider Eisenstein series on $G$ induced from unramified cusp forms on finite-dimensional Levi subgroups of maximal parabolic subgroups. Under a natural condition on maximal parabolic subgroups, we prove that the cuspidal Eisenstein series are entire on the full complex plane.
\end{abstract}

\maketitle

\section{Introduction}\label{sec:intro}

Meromorphic continuation of Eisenstein series on reductive groups, established in the seminal work of Langlands \cite{LA1}, has been a foundational basis for other developments in Langlands' program. Most notably, the Langlands--Shahidi method (e.g., \cite{KimSh, Kim}) in the study of $L$--functions originates from the idea of exploiting analytic properties of Eisenstein series together with the fact that $L$-functions appear in the Fourier coefficients of Eisenstein series.

The extension of the theory of Eisenstein series to Kac--Moody groups has been developed with anticipation of its potential roles in some of the central problems in number theory  \cite{BFH,Sh}. Indeed, starting with Garland's pioneering work \cite{G99, G04, G06}, there has been a significant body of work \cite{GMS1, GMS2, GMS3, GMS4, Li, GMP, G11} on Eisenstein series on affine Kac--Moody groups, including the function field case \cite{BK,Ka, P, LL}.   
Beyond affine Kac--Moody groups, Eisenstein series on rank 2 symmetric hyperbolic Kac--Moody groups were studied in \cite{CLL}. 
For a more general class of Kac--Moody groups, in the paper \cite{CGLLM}, the absolute convergence of Eisenstein series  is established for spectral parameters in the Godement range.

In the process of these developments, 
a striking difference between the finite-dimensional case and the affine case  was observed in  \cite{BK,GMP}, where the affine Eisenstein series induced from cusp forms on finite-dimensional Levi subgroups were shown to be {\it  entire} and not just meromorphic as they are in the finite-dimensional case. This phenomenon is not restricted to the affine case, and it was shown in  \cite{CLL} that cuspidal Eisenstein series on rank 2 symmetric hyperbolic Kac--Moody groups are also entire.

To understand the analytic properties of cuspidal Eisenstein series on Kac--Moody groups, one may naturally ask:~{\em For which parabolic subgroups are cuspidal Eisenstein series on  Kac--Moody groups entire?} For a general maximal parabolic subgroup, even in the finite dimensional case cuspidal Eisenstein series will not be entire, as is readily seen for  the real analytic Eisenstein series on $SL_n$.

In this paper, we give a natural condition  on parabolic subgroups (Property RD, below and Section~\ref{section-hol})
and show that  cuspidal Eisenstein series attached to a parabolic subgroup satisfying Property RD is holomorphic on the full complex plane.

More precisely, let $G$ be a representation-theoretic  Kac--Moody group, and let $\mathfrak g$ be the corresponding real Kac--Moody algebra with a fixed Cartan subalgebra $\mathfrak h$, Weyl group $W$ and the set $\{\alpha_i\}_{i \in I}$ of simple roots, where $I$ is the set of indices for the simple roots.  We assume that $\frak g$ is infinite-dimensional and non-affine since the other cases have been studied well as mentioned above.  

The real group $G_{\mathbb R}$ has
the Iwasawa decomposition $G_{\mathbb R}=UA^+K$ (\cite{DGH}), where $U$ is a maximal pro-unipotent subgroup, $A^+$ is the connected component of a maximal torus, and $K$ is a subgroup of $G_\R$ analogous to the maximal compact subgroup in the finite-dimensional theory. Denote by $\Gamma$ the arithmetic subgroup of $G_{\mathbb R}$ (Subsection~\ref{Zform}), and by  
$\mathcal{C}\subset\frak{h}$ the open fundamental chamber.  Using the exponential map $\exp: \mathfrak{h}\to A^+$, set $A_\mathcal{C}=\exp \mathcal{C}$.

Let $P$ be a maximal parabolic subgroup of $G_\R$ with fixed Levi decomposition and associated finite-dimensional Levi subgroup $M$, which is associated with a subset $\theta \subseteq I$. We denote by $\alpha_P$ the simple root  associated to the one element index in $I\backslash \theta$. Let $L$ be the derived subgroup of $M$.  
For a cusp form $f$ on $(L\cap\Gamma) \backslash L$, we recall that $f$ is unramified if $f$ is right invariant under the action of $L\cap K$.

For such an $f$, we define the Eisenstein series $E_f(s,g)$, $s \in \mathbb C$, $g \in G_\R$, in analogy with the classical case (see \eqref{def-cusp-E}).
In Proposition \ref{cusp}, we use the reduction mechanism of \cite{Bo,MW} to obtain absolute convergence of the cuspidal Eisenstein series $E_f(s,g)$ from that of Borel Eisenstein series established in \cite{CGLLM}.

Denote by $\varpi_P$ the fundamental weight associated to $\alpha_P$, and by $\rho_M$ the Weyl vector of $M$. Let
\[
W^\theta=\{w\in W: w^{-1}\alpha_i>0, i\in\theta\}.
\]
Then $P$ is said to satisfy {\em Property RD} if there exists a constant $D>0$, such that for every nontrivial element $w\in W^\theta$ we have
$$\langle D\varpi_P+\rho_M, \alpha^\vee\rangle \leq 0$$ for any positive root $\alpha$ such that $w^{-1} \alpha <0$. Property RD states that the coefficient of the simple root  
$\alpha_P$ grows faster than the coefficients of the simple roots in the subset $\theta$.
This property allows us to make use of the rapid decay of cusp forms on parabolic subgroups.

Now we state our main theorem.

\begin{theorem}\label{hol-1}
Let $f$ be an unramified cusp form on $(L\cap\Gamma) \backslash L$. If the maximal parabolic subgroup $P$ satisfies Property RD, then for any compact subset $\mathfrak{S}$ of $A_\mathcal{C}$, there exists a measure zero subset $S_0$ of $(\Gamma\cap U)\backslash U\mathfrak{S}$ such that $E_f(s,g)$ is an entire function of $s\in\mathbb{C}$ for 
$g\in (\Gamma\cap U)\backslash U\mathfrak{S}K- S_0K$.
\end{theorem}

Here a measure zero set appears because absolute convergence is only established almost everywhere for the Eisenstein series in general. In the setting of everywhere convergence (as in \cite{CGLLM}), the measure zero set is not needed. See Theorem \ref{hol-2} for more precise statements.

The main idea of the proof is to exploit rapid decay of a cusp form on the maximal parabolic subgroup, guaranteed by Property RD. This property may be considered as extraction of the essential properties of maximal parabolic subgroups used in the earlier results on affine Kac--Moody groups \cite{GMP} and on rank 2 hyperbolic groups \cite{CLL}.

In Section \ref{sec-ample} we show that a large class of Kac--Moody groups  have parabolic subgroups satisfying Property RD, including
 the Kac--Moody group $G$ associated with the Feingold--Frenkel rank $3$ hyperbolic Kac--Moody algebra \cite{FF}.  We note that the corresponding Kac--Moody group did not satisfy the conditions in \cite{CGLLM} that allowed the authors to prove convergence of Eisenstein series.
 
 It would be an interesting question for future investigation to characterize
 the full class of Kac--Moody groups that admit parabolic subgroups satisfying Property RD.

\subsection*{Acknowledgments}
We are grateful to Stephen D. Miller for his useful comments, and would like to thank the anonymous referees for helpful comments. 

\section{Kac--Moody groups and Borel Eisenstein series} 

In this section we recall the construction of Kac--Moody groups and then state the convergence results for Borel Eisenstein series proven in \cite{CGLLM} and used in this work. 
We try to keep the exposition self-contained as much as possible, and refer other details to {\it loc. cit.} and the references therein. 

\subsection{Kac-Moody groups}\label{Zform}
Let $I=\{ 1, 2, \dots , r \}$,
${\sf A}=(a_{ij})_{i,j \in I}$ be an $r \times r$  symmetrizable  generalized Cartan matrix, and $(\mathfrak h_{\mathbb C}, \Delta, \Delta^\vee)$ be a realization of ${\sf A}$, where $\Delta=\{ \alpha_1, ... , \alpha_r \} \subset \mathfrak h_{\mathbb C}^*$ and  $\Delta^\vee=\{ \alpha^\vee_1, ... , \alpha^\vee_r \} \subset \mathfrak h_{\mathbb C}$ are the set of simple roots and set of simple coroots, respectively.

Following \cite{CGLLM}, we assume that ${\sf A}$ is nonsingular, so that $\mathfrak{h}_{\mathbb C}$ and $\mathfrak{h}_{\mathbb C}^*$  are spanned by the simple roots $\alpha_i$ and simple coroots $\alpha_i^\vee$, respectively. Recall that  $\langle\alpha_j, \alpha^\vee_i\rangle = a_{ij}$ for $i,j \in I$, where $\langle \cdot,\cdot\rangle$ is the natural pairing between $\frak{h}_{\mathbb C}^*$ and $\frak{h}_{\mathbb C}$.
Denote the fundamental weights by $\varpi_i\in\mathfrak{h}_{\mathbb C}^*$, $i\in I$, which form the   basis of $\frak h_{\mathbb C}^*$ dual to the $\alpha_i^\vee$. 

Let $\mathfrak g_{\C}=\mathfrak g_{\C}({\sf A})$ be the Kac--Moody algebra associated to $(\mathfrak h, \Delta, \Delta^\vee)$, which we assume to be infinite-dimensional 
throughout this paper.  We denote by $\Phi$ the set of roots of $\mathfrak g_{\C}$ and have $\Phi = \Phi_+ \sqcup \Phi_{-}$, where $\Phi_+$ (resp. $\Phi_{-}$) is the set of positive (resp. negative) roots corresponding to the choice of $\Delta$.  Let $w_i:=w_{\a_i}$ denote the simple Weyl reflection associated to the simple root $\alpha_i$. The $w_i$ for $i\in I$  generate the Weyl group $W$ of $\frak g_\C$.  

Let $\frak{g}$ be the real Lie subalgebra generated by the Chevalley generators $e_i$ and $f_i$ for $i\in I$, so that $\frak g_\C=\frak g\otimes_\R \C$. As is standard (see \cite[\S4]{CG}), we can define the $\mathbb{Z}$-form $\mathcal{U}_{\mathbb{Z}}$ of the universal enveloping algebra $\mathcal{U}_{\mathbb{C}}$ of $\frak{g}$ using the Chevalley generators.

We now review the representation theoretic Kac--Moody groups  $G_{F}$ associated to $\frak g$ and a field $F$ of characteristic zero, following  \cite{CG}. 
Let $(\pi,V)$ denote the unique irreducible highest weight module for $\mathfrak{g}$ corresponding to a choice of some regular dominant integral weight, and
fix a nonzero highest weight vector $v \in V$.   Then we have a lattice 
$$V_{\mathbb{Z}}\  =\ \mathcal{U}_{\mathbb{Z}}\cdot v,$$
and we put $V_{F}=F\otimes_{\mathbb{Z}}V_{\mathbb{Z}}$.  Since $V$ is integrable, it makes sense to define for  $s,t\in F$
 and $i\in I$ the operators
$$ \ \ \ \ \ \ \ u_{\alpha_i}(s)\ =\ \exp(\pi(se_i))$$
$$\text{and}~~u_{-\alpha_i}(t)\ =\ \exp(\pi(tf_i)),$$
  which act locally finitely on $V_F$ and thus define elements of  $\textrm{Aut}(V_{F})$.

For $t \in F^{\times}$ and $i\in I$  we set \[ w_{\alpha_i}(t) \  = \ u_{\alpha_i}(t)  u_{-\alpha_i}(-t^{-1}) u_{\alpha_i}(t)\] and define \[ h_{\alpha_i}(t)\ = \  w_{\alpha_i}(t) w_{\alpha_i}(1)^{-1} .\]
Each simple root $\alpha_j$ defines a character on $\{ h_{\alpha_i}(t)\ :\ t\in  F^{\times}\}$ by
\begin{equation}\label{halphaialphajchar}
  h_{\alpha_i}(t)^{\alpha_j} \ = \ t^{\langle \alpha_j,\alpha_i^\vee\rangle}.
\end{equation}
The subgroup $\langle w_{\alpha_i}(1):i\in I\rangle$ of $\textrm{Aut}(V_{F})$   contains a full set of Weyl group representatives.
For a real root $\alpha$, choose once for all an expression  
  $\alpha=w_{\beta_1}\cdots w_{\beta_\ell}\alpha_i$ for some $i\in I$ and $\beta_1,\ldots,\beta_\ell\in\Delta$, and define the
corresponding one-parameter subgroup 
\[
  u_\alpha(s) \ \ = \ \ w u_{\alpha_i}( s)w^{-1} \in \textrm{Aut}(V_{F}) \qquad (s \in F),
\]
where   $w =w_{\beta_1}(1)\cdots w_{\beta_\ell}(1)$. 

Define $$G^0_{F}\ = \ \langle u_{\alpha_i}(s), u_{-\alpha_i}(t) : s,t\in F, \ i \in I \rangle\ \subset \ \textrm{Aut}(V_F).$$
Following  \cite[\S5]{CG},  we choose a  coherently ordered basis    $\Psi$ of $V_\Z$, and denote by $B^0_F$ the subgroup of $G^0_F$ consisting of  elements which act upper-triangularly  with respect to $\Psi$. The basis $\Psi$ induces certain topologies on $B^0_F$ and $G^0_F$, and we let $B_F$ and $G_F$ denote their completions respectively. 
 When the field $F$ is unspecified, $B$ should be interpreted as $B_\R$.

The following subgroups of $G_F$ are of particular importance to us:
 \begin{itemize}
 \item  $A_F=\langle h_{\alpha_i}(s) : s\in F^\times, i \in I\rangle$; and
 \item  $U_F\subset B_F$ is defined exactly as $B_F$, but with the additional stipulation that elements act unipotently upper-triangularly with respect to $\Psi$.  It contains all subgroups parameterized by the $u_\alpha(\cdot)$, where $\alpha\in \Phi_+$ is a real root.  Then $B_F=U_FA_F=A_FU_F$.  
 \end{itemize}
When no subscript is given, $A$ and $U$ should be interpreted as $A_\R$ and $U_\R$.
We also have the following  subgroups specific to the situation $F=\R$:
 \begin{itemize}
\item $K$ is the subgroup of $G_\R$ generated by all  $\exp(t(e_i-f_i))$, $t\in \R$ and $i\in I$ (see \cite{KP}); and
    \item   $A^+=\langle h_{\alpha_i}(s) : s\in\mathbb{R}_{>0}, i \in I\rangle$.  In fact, $(\mathbb{R}_{>0})^r$ can be identified with $A^+$ via the isomorphism  $(x_1,\dots, x_r)\mapsto h_{\alpha_1}(x_1)\cdots h_{\alpha_r}(x_r)$, under which $A^+$ has the Haar measure $da$ corresponding to $\prod_{i=1}^r \frac{dx_i}{x_i}$.
\end{itemize}

\begin{theorem}[\cite{DGH}] \label{Iwasawa}
We have the Iwasawa decomposition
\begin{equation}\label{iwasawa}
G_\R \ = \  UA^+K
\end{equation} with uniqueness of expression.
\end{theorem}

\begin{remark}
In the paper \cite{DGH}, twin BN-pairs are used to prove the Iwasawa decomposition for the group $G^0$. The Iwasawa decomposition for $G_\R$ follows, since
\[ G_\R = \bigcup_{w \in W} B w B= \bigcup_{w \in W} B w U_w=BG^0=BB^0 K=BK = UA^+ K, \]
where $U_{w}:= U \cap w^{-1} B^0_- w \subset G^0$ with the opposite Borel denoted by $B^0_-$.  
\end{remark}

Let $u(g)$, $a(g)$, and $k(g)$ denote the projections   from $G_{\R}$ onto  each of the respective factors in (\ref{iwasawa}).
We define the discrete group
$\Gamma :=G_\R\cap \mathrm{Aut}(V_{\mathbb{Z}})=\{\gamma\in G_\R : \gamma\cdot V_{\mathbb{Z}}= V_{\mathbb{Z}}\}.$
As in \cite{G04}, it can be shown that $(\Gamma\cap U) \backslash U$ is the projective limit of a projective family of finite-dimensional compact nil-manifolds  and thus admits
a projective limit measure $du$ which is a right $U$-invariant probability measure.

\subsection{Borel Eisenstein series}\label{constant-term}

Using the identification of $A^+$ with $(\R_{>0})^r$, each element of $\frak h_{\mathbb C}^*$ gives rise to a quasicharacter of $A^+$. Let $\lambda\in \frak h_{\mathbb C}^*$ and let $\rho\in\mathfrak{h}_{\mathbb C}^*$ be  the Weyl vector, which is characterized by $\langle\rho,\alpha_i^\vee\rangle =1$, $i\in I$.  We set
  \begin{equation}\label{Philambdadef}
    \aligned
    \Phi_{\lambda}:& \ G_\R\to {\mathbb C}^{\times}, \\
    & \ g  \mapsto a(g)^{\lambda+\rho},
    \endaligned
  \end{equation}
   which is well-defined by the uniqueness of the Iwasawa decomposition. 

 Let $B$ and $\Gamma$ be   as defined in Section \ref{Zform}.
 Define the Borel Eisenstein series on $G_\R$ to be the infinite formal sum
\[
E_{\lambda}(g)\quad \defeq \sum_{\gamma\in (\Gamma\cap{B})\backslash \Gamma}
\Phi_{\lambda}(\gamma g).
\]
We define for all $g\in G_\R$ the so-called ``upper triangular'' constant term\footnote{This is to be distinguished from the proposal in \cite{BK} to consider constant terms in uncompleted, or so-called ``lower triangular'', parabolics.}
\[
E^\sharp_\lambda(g)\ \defeq \ \int_{(\Gamma\cap U)\backslash U}E_\lambda(ug)du,
\]
which is left $U$-invariant and right $K$-invariant. It is immediate from the definition that $E^\sharp_\lambda(g)$ is determined by the $A^+$-component $a(g)$ of $g$ in the Iwasawa decomposition.

We now state the convergence results for the constant term $E^\sharp_\lambda(g)$ as well as the Eisenstein series $E_\lambda(g)$ itself obtained in \cite{CGLLM}. We will use these results later in this paper.


Let $\mathcal{C} \subset\mathfrak{h}$ be the open fundamental chamber
\[
\mathcal{C} \ = \ \{x\in \mathfrak{h}: \langle\alpha_i, x\rangle>0,~  i\in I\}.
\]
   Let $\mathfrak{C}$ denote the interior of the Tits cone $\mathop{\bigcup}_{w \in W}  w\overline{  \mathcal C}$   corresponding to $\mathcal{C}$ \cite[\S3.12]{K}.  Using the exponential map $\exp: \mathfrak{h}\to A^+$, set $A_\mathcal{C}=\exp \mathcal{C}$ and $A_\mathfrak{C}=\exp \mathfrak{C}$.
Let
\begin{equation} \label{C-star}
\mathcal{C}^* \ = \ \{\lambda\in\mathfrak{h}^*: \langle\lambda,\alpha_i^\vee\rangle>0, ~ i\in I\}.
\end{equation}

Based on a crucial lemma due to Looijenga \cite{Lo} together with some elementary estimate on the Riemann $\zeta$-function, the following theorem is proven in \cite{CGLLM}.  

\begin{theorem}\cite{CGLLM} \label{3.3}
If   $\lambda\in\mathfrak{h}^*_\mathbb{C}$ satisfies $\operatorname{Re}(\lambda  - \rho) \in \mathcal{C}^*$, then
$E^\sharp_\lambda(g)$ converges absolutely for $g\in UA_\mathfrak{C}K$, and in fact uniformly for $a(g)$ lying in any fixed compact subset of $A_\mathfrak{C}$.
\end{theorem}

Consequently, a simple application of  Tonelli's theorem as in \cite[\S9]{G04} yields the following result.

\begin{corollary}\cite{CGLLM}\label{cor}
For $\lambda\in\mathfrak{h}^*_\mathbb{C}$ with $\operatorname{Re}(\lambda -\rho) \in \mathcal{C}^*$ and any  compact subset $\mathfrak{S}$ of $A_\mathfrak{C}$, there exists a measure zero subset $S_0$ of $(\Gamma\cap U)\backslash U\mathfrak{S}$ such that the series $E_\lambda(g)$ converges absolutely for $g\in (\Gamma\cap U)\backslash U\mathfrak{S}K- S_0K$.
\end{corollary}

Moreover, the everywhere convergence of $E_\lambda(g)$ over $\Gamma UA_\frak{C} K$ was established under the following combinatorial property, by exploring the idea that the Borel Eisenstein series can be nearly bounded by its constant term. 
We set
\begin{equation}\label{Phiwdef}
  \Phi_w \ \defeq \ \Phi_+\cap w^{-1}\Phi_{-}, \quad w\in W.
\end{equation}

The following property was introduced in \cite{CGLLM} and will be used in Theorem~\ref{hol-2}.
\begin{property}\label{conj}
Every nontrivial $w \in W$ can be written as
$w=vw_{\beta}$, where $\beta$ is a positive simple root, $\ell(v)<\ell(w)$, and $\alpha -\beta$ is never a real root for any $\alpha \in \Phi_v$.
\end{property}

In \cite{CGLLM} it is shown that Property \ref{conj} holds when the Cartan matrix ${\sf A}= (a_{ij})$ is symmetric with $|a_{ij}| \ge 2$ for all $i,j$; for rank 2 hyperbolic groups, the condition $|a_{ij}| \ge 2$ is sufficient and the matrix $A$ does not need to be symmetric. The main result  of {\it loc. cit.} is the following.

\begin{theorem}\cite{CGLLM} \label{thm-CGL+}
Assume  that $\lambda\in\mathfrak{h}^*_\mathbb{C}$ satisfies $\operatorname{Re}(\lambda -\rho) \in \mathcal{C}^*$, and that Property~\ref{conj} holds. Then the Kac--Moody Eisenstein series $E_\lambda(g)$ converges absolutely for $g\in \Gamma UA_\mathfrak{C}K$.
\end{theorem}

\section{Cuspidal Kac--Moody Eisenstein series}\label{KMEseries}

\subsection{Definition of Eisenstein series} For any subset $\theta \subset I$, denote by $W_\theta$ the subgroup of $W$ generated by the reflections $w_i$ for $i \in \theta$. Assume that we choose  $\theta$ such that $W_\theta$ is a finite group.
Define the parabolic subgroup of $G_{\mathbb R}$ associated to $\theta$ by
\[ P_\theta \ \defeq \  B W_\theta B\, , \]
where each $w\in W_\theta$ is identified with one of its representatives in $G_{\mathbb R}$ as in Section~\ref{Zform}. For the rest of this section, we will fix $\theta$ and   write $P=P_\theta$.

Let $N$ be the pro-unipotent radical of $P$, and  $M$  be  the subgroup of $P$ generated by $A$ and 
 $u_{\pm \alpha_i}(t)$, $i\in \theta$, $t\in \mathbb{R}$. Then $M$ is a finite-dimensional  real reductive group. By \cite[6.1.13]{Ku}, we have the Levi decomposition $P=M\ltimes N$.

We shall introduce certain subgroups and explain some  structures of $M$.  
Let $L$ be the subgroup of $M$ generated by $u_{\pm\alpha_i}(t)$, $i \in \theta$, $t\in \mathbb{R}$. Then $L$ is a finite dimensional semisimple group. Let $A_1 \defeq A\cap L$ be the subgroup of $A$ generated by $h_{\alpha_i}(t)$, $t\in \mathbb{R}^\times$, $i \in \theta$, which is a  maximal split torus of $L$. Set 
\[
A_1^+ \ = \ \langle h_{\alpha_i}(t): t\in \mathbb{R}_{>0}, i \in \theta\rangle \ \cong  \ A_1/(A_1\cap K).
\] 
We also introduce
$$H \defeq  \{a\in A: a^{\alpha_i}=\pm 1, i \in \theta\},$$
and let $H^+$ be the identity component of $H$, so that $H^+\cong H/(H\cap K)$.
 
 \begin{lemma} \label{lem-a}
We have $A\cong A_1^+\times H$ and $A^+\cong A_1^+\times H^+$.
\end{lemma}

\begin{proof}
The second isomorphism follows from taking the identity component of the first. Let us prove the first one. By the assumption on $\theta$,  
the natural pairing between $\bigoplus_{i\in \theta}\mathbb{Z}\alpha_i$ and $\bigoplus_{i\in \theta}\mathbb{Z}\alpha_i^\vee$ is non-degenerate. From this it follows easily that
$A_1^+\cap H=\{1\}$. For the same reason, for any $a\in A$, there exists $a_1\in A_1^+$ such that $a_1^{\alpha_i}=|a^{\alpha_i}|$ for all $i\in \theta$. Then clearly 
$aa_1^{-1}\in H$, which proves that $A=A_1^+H$ and therefore $A\cong A_1^+\times H$. 
\end{proof}

\begin{remark} \label{rm-levi}
By Lemma \ref{lem-a} it is clear that $M=LH$, and we can say a bit more about this decomposition. Note that $H^+$ is the identity component of the center of $M$.
By  \cite[4.3.1]{Bo}, we have
\[
M\cong {}^\circ M\times H^+,
\]
where
\[
{}^\circ M:=\bigcap_{\chi\in X(M)}\textrm{ker }|\chi|.
\]
Here $X(M)$ denotes the lattice of algebraic characters of $M$ defined over $\mathbb{R}$.
Since $L$ is semisimple, we have $L\subset {}^\circ M$ hence $L\cap H\subset {}^\circ M\cap H$, which is a finite group contained in $K$.
\end{remark}

It will be convenient to rewrite $a(g)$ as  
\[
\textrm{Iw}_{A^+}: G_{\R}\to A^+,\quad g\mapsto a(g)
\]
 from the Iwasawa decomposition $G_{\R}=UA^+K$  in  \eqref{iwasawa}.
The parabolic subgroup $P$ gives rise to the decomposition $G_{\R}=NMK$, which unlike the Iwasawa decomposition \eqref{iwasawa} is not unique since $M\cap K$ is nontrivial.  However in Lemma \ref{welldef} below we will show that the projection
\[
\textrm{Iw}_M: G_{\R}\to M/(M\cap K)
\]
is well-defined. 
Although $M=LH$ is in general not a direct product (cf. Remark \ref{rm-levi}),  we will also deduce from the fact $L\cap H\subset K$ that the maps
\[
\textrm{Iw}_L: G_{\R}\to L/(L\cap K) \quad \text{ and } \quad \textrm{Iw}_{H^+}: G_{\R}\to H^+,
\]
which are induced by the projections onto the  factors  $L$ and $H$ of $M$, respectively,
are well-defined.  
Similarly, using the Iwasawa decomposition for $L$, we may define the map
\[
\textrm{Iw}_{A_1^+}: L/(L\cap K)\to A_1^+.
\]

\begin{lemma} \label{welldef}
The Iwasawa maps $\mathrm{Iw}_M$, $\mathrm{Iw}_L$ and $\mathrm{Iw}_{H^+}$ are well-defined.
\end{lemma}

\begin{proof}
Well-definedness of $\textrm{Iw}_M$: Assume that $nmk=m'$, where $n\in N$, $m, m'\in M$ and $k\in K$. We need to show that $m'\in mK$, or that $n=1$. Apply the Iwasawa 
decomposition of $M$ to write 
$m=u_Ma_Mk_M$ and $m'=u_M'a_M'k_M'$, where $u_M, u_M'\in U_M:=U\cap M$, $a_M, a_M'\in A^+$ and $k_M, k_M'\in K\cap M$. By the uniqueness of the Iwasawa decomposition 
$G=UA^+K$, we have in particular that
\[
n u_M=u_M',
\]
which forces that $n=1$, $u_M=u_M'$ because $U=U_M\ltimes N$. Hence $m'=mk$, which implies that $\textrm{Iw}_M$ is well-defined.

Well-definedness of  $\textrm{Iw}_L$ and $\textrm{Iw}_{H^+}$: In Remark \ref{rm-levi} we have shown that $M=LH$ and $L\cap H\subset K$. Since we already 
proved that $\textrm{Iw}_M$ is well-defined, we now only need to consider 
\[
l h =l' h'\in M,
\]
where $l, l'\in L$ and $h, h'\in H$. Then $l^{-1}l'=h(h')^{-1}=(h')^{-1}h\in L\cap H\subset K$, hence $l'\in l K$, $h\in h' K$, which prove that
$\textrm{Iw}_L$ and $\textrm{Iw}_{H^+}$ are both well-defined.
\end{proof}

We have the following decomposition of the Iwasawa map $\textrm{Iw}_{A^+}$.

\begin{lemma} We have
\begin{equation}\label{iwa-3}
\mathrm{Iw}_{A^+}=(\mathrm{Iw}_{A_1^+}\circ\mathrm{Iw}_L)\times \mathrm{Iw}_{H^+}: G_{\R} \longrightarrow A^+\cong A_1^+\times H^+.
\end{equation}
\end{lemma}

\begin{proof}
We can verify \eqref{iwa-3} by tracking the definitions.  Assume that
\[
g=nmk,\quad m=lh,
\]
where each factor belongs to the subgroup denoted by the corresponding capital letter. Applying the Iwasawa decomposition of $L$ to write
\[
l=u a_1 k_L,
\]
where $u\in U_M$, $a_1\in A_1^+$ and $k_L\in K\cap L$. We can also write $h=h^+ k_H$, where $h^+\in H^+$ and $k_H\in K\cap H$. Since $H^+$ commutes
with $L$, we have
\[
g=(n u)(a_1 h^+)(k_L k_H k),
\]
which is the Iwasawa decomposition according to $G_\mathbb{R}=UA^+K$. By the definitions of various Iwasawa maps, we have
\begin{align*}
\textrm{Iw}_{A^+}& =a_1h^+ \\
& =\textrm{Iw}_{A_1^+}(l K_L)\cdot \textrm{Iw}_{H^+}(g) \\ 
& = \textrm{Iw}_{A_1^+}(\textrm{Iw}_L(g)) \cdot \textrm{Iw}_{H^+}(g),
\end{align*}
which proves \eqref{iwa-3}.
\end{proof}

Consider $\lambda_P = \sum_{i \not\in \theta} s_i \varpi_i \in \mathfrak h_{\mathbb C}^*$, $s_i \in \C$, where $\varpi_i$ are the fundamental weights.
As before, let $\rho\in \mathfrak{h}_\C^*$ be the Weyl vector, i.e.,  $\langle\rho,\alpha_i^\vee\rangle=1$, $i\in I$.
Likewise, let   $\rho_M$ be half the  sum of the positive roots of $M$, and set $\rho_P \defeq \rho-\rho_M$.
Since $\lambda_P$ is clearly trivial on $A_1=A\cap L$, the decomposition (\ref{iwa-3}) implies that
\begin{equation}\label{extra}
\textrm{Iw}_{A^+}(\cdot)^{\lambda_P}=\textrm{Iw}_{H^+}(\cdot)^{\lambda_P}.
\end{equation}
Similarly, every root $\beta$ of $L$ is trivial on $H^+$, hence
\begin{equation}\label{levi}
\textrm{Iw}_{A^+}(\cdot)^\beta=\left ( \textrm{Iw}_{A_1^+}\circ\textrm{Iw}_L(\cdot) \right )^\beta.
\end{equation}

For any $\lambda_P= \sum_{i \not\in \theta} s_i \varpi_i \in \mathfrak h_\C^*$, define the auxiliary Eisenstein series
\begin{equation}\label{auxEis}
E(\lambda_P,g)\defeq \sum_{\gamma\in (\Gamma\cap P)\backslash\Gamma}\textrm{Iw}_{H^+}(\gamma g)^{\lambda_P+\rho_P}, \quad g \in G_\R.
\end{equation}
More generally, for a bounded function $f$ defined on $M/M\cap K$, define the Eisenstein series $E_f(\lambda_P,g)$ by
\begin{equation} \label{def-cusp-E}
E_f(\lambda_P,g)\defeq \sum_{\gamma\in (\Gamma\cap P)\backslash\Gamma}\textrm{Iw}_{H^+}(\gamma g)^{\lambda_P+\rho_P}\, f\left(\textrm{Iw}_M(\gamma g)\right), \quad g \in G_\R.
\end{equation}
Then (\ref{auxEis}) corresponds to the special case that $f\equiv 1$.

\subsection{Convergence by the reduction mechanism}\label{reduction} The idea of a reduction mechanism is to bound general Eisenstein series by Borel ones. It was first suggested by Bernstein to Borel, and used by Garland in \cite{G11} for the affine case, who was inspired by some results in \cite{GMRV}.
Following this idea, we prove the absolute convergence of parabolic Eisenstein series $E_f(\lambda_P,g)$ in the Godement range. Similar proofs for the finite-dimensional case can be found in \cite[Proposition 12.6]{Bo} and
\cite[Proposition II.1.5]{MW}, which also use Borel Eisenstein series to bound parabolic ones, but the arguments are slightly different.

\begin{proposition} \label{cusp}
Let $f$ be a bounded function on $M/M\cap K$. For $\lambda_P = \sum_{i \not\in \theta} s_i \varpi_i \in \mathfrak h_\C^*$, with $\operatorname{Re}(s_i)>\langle \rho_P, \alpha_i^\vee\rangle$ , $i \not\in \theta$, and any  compact subset $\mathfrak{S}$ of $A_\mathfrak{C}$, there exists a measure zero subset $S_0$ of $(\Gamma\cap U)\backslash U\mathfrak{S}$ such that the series $E_f(\lambda_P, g)$ converges absolutely for $g\in (\Gamma\cap U)\backslash U\mathfrak{S}K- S_0K$.
Moreover, $S_0$ can be chosen to work uniformly for all $s_i$, $i\not\in\theta$, in right half-spaces
of the form $\operatorname{Re}(s_i)>\sigma_i$, where $\sigma_i>\langle \rho_P, \alpha_i^\vee\rangle$ are fixed. 
\end{proposition}

\begin{proof}
Since the function $f$ is assumed to be bounded, it suffices to prove the almost everywhere absolute convergence of the auxiliary Eisenstein series $E(\lambda_P,g)$, and we may assume that $s_i>\langle\rho_P, \alpha_i^\vee\rangle$ are real numbers for $i \not\in \theta$. Let $\varepsilon>0$ be sufficiently small such that
\[
s_i>(1+\varepsilon)\langle\rho_P,\alpha_i^\vee\rangle
\]
for $i\not\in\theta$.
We put 
\[
\lambda=(1+\varepsilon)\rho_M +\lambda_P.
\]
Then $\lambda-\rho\in\mathcal{C}^*$, where
$\mathcal C^*$ is defined in \eqref{C-star}.
Indeed, $\langle\lambda,\alpha_i^\vee\rangle=1+\varepsilon$ for $i\in\theta$, and
\[
\langle\lambda, \alpha_i^\vee\rangle=(1+\varepsilon)\langle\rho_M,\alpha_i^\vee\rangle+s_i>  (1+\varepsilon)\langle\rho_M+\rho_P, \alpha_i^\vee\rangle=1+\varepsilon
\]
for $i\not\in\theta$.

Consider the Borel Eisenstein series
\[
E_\lambda(g)=\sum_{\gamma\in (\Gamma\cap B)\backslash \Gamma}\Phi_\lambda(\gamma g)=\sum_{\gamma\in (\Gamma\cap P)\backslash \Gamma}\sum_{\delta\in (\Gamma\cap B)\backslash \Gamma\cap P}\Phi_\lambda(\delta\gamma g).
\]
By the Iwasawa decomposition $G_\R=NMK$ and Lemma \ref{lem3.6} below, the last inner sum is equal to 
\[
\sum_{\delta\in (\Gamma_M\cap B_M)\backslash \Gamma_M}\Phi_\lambda(\delta \cdot \textrm{Iw}_M(\gamma g)),
\]
where $\Gamma_M:=\Gamma\cap M$ and $B_M:=B\cap M$. By our choice of $\lambda$, it can be further manipulated to give
\begin{align*}
&\textrm{Iw}_{H^+}(\gamma g)^{\lambda_P+\rho_P} \sum_{\delta\in (\Gamma_M\cap B_M)\backslash \Gamma_M}\Phi_\lambda(\delta\cdot\textrm{Iw}_{L}(\gamma g))\\
=& \textrm{Iw}_{H^+}(\gamma g)^{\lambda_P+\rho_P} \sum_{\delta\in (\Gamma_M\cap B_M)\backslash \Gamma_M}\textrm{Iw}_{A_1^+}(\delta\cdot\textrm{Iw}_{L}(\gamma g))^{\lambda+\rho}.
\end{align*}

Since $\lambda$ and $\rho$ restrict to $(1+\varepsilon)\rho_M$ and $\rho_M$ on $A_1^+$ respectively,  the last sum is a Borel Eisenstein series  on $M$ with spectral parameter $(1+\varepsilon)\rho_M$, which is in the Godement range. Since the spectral parameter is real, the series is a sum of positive terms, and  hence is never zero. Moreover, it has a positive lower bound. This fact should be well-known,  but it does not seem to be widely available in the literature. We supply a proof below for completeness.

 Recall that a Siegel domain for $M$ is of the form
\[
\frak{S}_{t, \omega}:=\omega A_t K_M,
\]
where $\omega$ is a compact subset of $B_M$, $K_M=K\cap M$ and 
\[
A_t:=\{a\in A^+: a^{\alpha_i}>t, i \in \theta\}
\]
for $t>0$. 
By classical reduction theory (see e.g. \cite[I.2.1]{MW}), we have
$
M=\Gamma_M \cdot \frak{S}_{t, \omega}
$
for $\omega$ sufficiently large and $t$ sufficiently small.  Write $m\in M$ accordingly as
\[
m=\gamma x a k,
\]
where $\gamma\in \Gamma_M$, $x \in \omega$, $a\in A_t$ and $k\in K_M$. The Borel Eisenstein series $E_{(1+\varepsilon)\rho_M}(\cdot)$ on $M$ takes value
\begin{align*}
E_{(1+\varepsilon)\rho_M}(m) & =E_{(1+\varepsilon)\rho_M}(x a)\geq \Phi_{(1+\varepsilon)\rho_M}(x a)=a^{(2+\varepsilon)\rho_M}\Phi_{(1+\varepsilon)\rho_M}(x)\\
& \geq \kappa:=t^{(2+\varepsilon)C_M} \cdot h_\omega>0,
\end{align*}
where $C_M$ is the sum of the coefficients of $\alpha_i$ in $\rho_M$ and $h_\omega$ is the minimum value of  the positive continuous function $\Phi_{(1+\varepsilon)\rho_M}(\cdot)$ on the compact set $\omega$. 

With this lower bound established, it follows that
\[
E(\lambda_P,g)=\sum_{\gamma\in (\Gamma\cap P)\backslash \Gamma}\textrm{Iw}_{H^+}(\gamma g)^{\lambda_P+\rho_P} \leq \kappa^{-1}E_\lambda(g).
\]
The proposition follows from this comparison and Corollary \ref{cor}.
\end{proof}

\begin{lemma} \label{lem3.6}
If $\gamma\in \Gamma\cap P$ decomposes as $\gamma=nm$, where $n\in N$, $m\in M$, then $n, m \in \Gamma$.
\end{lemma}

\begin{proof}
Recall from \cite[\S5]{CG} that a coherently ordered basis $\Psi$ of $V_\mathbb{Z}$ consists of weight vectors. If $m \not\in \Gamma$, then there exists $v_0\in \Psi$ such that $m\cdot v_0\not\in V_\mathbb{Z}$. Assume that $v_0$ is of weight $\mu$. Then there exists a finite subset $\Psi_0\subset \Psi$ such that each $v\in \Psi_0$ is of weight $\mu+\alpha$ for some $\alpha$ in the root lattice of $M$, and that $m\cdot v_0$ is of the form
\[
m \cdot v_0=   \sum_{v \in \Psi_0} c_v   v \qquad\text{ with }  c_v \not \in \mathbb Z \text{ for some } v \in \Psi_0. 
\] 
It is clear that  $nm\cdot v_0$ is of the form
\[
nm\cdot v_0 = m \cdot v_0 +  \sum_{v \in \Psi \setminus \Psi_0} c_v v  \not \in V_\mathbb Z,
\]
which contradicts that $\gamma=nm \in \Gamma$. 
\end{proof}

\section{Holomorphy of cuspidal Eisenstein series} \label{section-hol}

In this section, we prove our main result. We keep the assumption that $\mathfrak g$ is infinite-dimensional with nonsingular Cartan matrix $\sf A$. Let $P=P_\theta$ be a maximal parabolic subgroup of $G_\R$ with  fixed Levi decomposition and associated finite-dimensional split  reductive Levi subgroup $M$ as in Section \ref{KMEseries}. We set $\{i_P\}:=I\setminus\theta$  and recall that $\alpha_P=\alpha_{i_P}$, and $\varpi_P$ is the fundamental weight for $\alpha_P$. Let $W^\theta$ be the set of minimal length representatives of $W_\theta\backslash W$, i.e.
\[
W^\theta \defeq \{w\in W: w^{-1}\alpha_i>0, i\in\theta\}.
\]

\begin{definition}\label{ample}
A maximal parabolic subgroup $P=P_\theta$ whose Levi subgroup $M$ is a finite-dimensional split reductive group, is said to satisfy {\it Property RD} if there exists a constant $D>0$, such that for every nontrivial element $w\in W^\theta$ we have
$\langle D\varpi_P+\rho_M, \alpha^\vee\rangle \leq 0$ for any  $\alpha\in\Phi_{w^{-1}}$.
\end{definition}

\begin{remark} \label{rmk-sh}
Note that $\langle \varpi_P, \alpha^\vee\rangle>0$ for any $\alpha\in \Phi_{w^{-1}}$ with $w\in W^\theta$. Hence if there exists $D$ that satisfies the condition in Definition 
\ref{ample}, then any positive constant smaller than $D$ will also satisfy the condition. Actually, one can see that $P_\theta$ has Property RD if and only if
for every nontrivial element $w\in W^\theta$ we have
$\langle \rho_M, \alpha^\vee\rangle < 0$ for any  $\alpha\in\Phi_{w^{-1}}$.
\end{remark}

The fact $L\cap H\subset K$ allows us to view a function $f$ on $L/L\cap K$ as defined on $M/M\cap K$. In particular an unramified cusp form $f$ on $L$ can be regarded as a bounded function on $M/M\cap K$, hence one may define the Eisenstein series $E_f(\lambda_P,g)$ as in Section \ref{KMEseries}.  In the sequel we shall adopt this interpretation without further remark. Moreover, since $\lambda_P= s \varpi_P$ for some $s \in \C$, we write $E_f(s, g)= E_f(\lambda_P,g)$ in the rest of the paper.

We state the main theorem again.  
\begin{theorem}[Theorem \ref{hol-1}]\label{hol}
Let $f$ be an unramified cusp form on $L\cap\Gamma \backslash L$. If the maximal parabolic subgroup $P$ satisfies Property RD, then for any compact subset $\mathfrak{S}$ of $A_\mathcal{C}$, there exists a measure zero subset $S_0$ of $(\Gamma\cap U)\backslash U\mathfrak{S}$ such that $E_f(s,g)$ is an entire function of $s\in\mathbb{C}$ for $g\in (\Gamma\cap U)\backslash U\mathfrak{S}K- S_0K$.
\end{theorem}

The rest of this section is devoted to the proof of Theorem \ref{hol}. We shall follow the strategy of \cite{GMP} and \cite{CLL} with Property RD at hand. We prove an inequality in Lemma~\ref{ine} which is crucial in our proof of the main theorem. This inequality enables us to deduce the entirety of 
Eisenstein series from the rapid decay of cusp forms. Along the way, we note some interesting properties of the infinite-dimensional Kac-Moody algebras that are considered here.

The following lemma is essentially the same as  \cite[Lemma 3.2]{GMP}. Write $x^y=yxy^{-1}$ for $x,y\in G_\R$, and set $U_w = \prod_{\alpha \in \Phi_{w}} U_\alpha$ for $w \in W$ where $U_\alpha$ is the one-parameter subgroup corresponding to $\alpha$.

\begin{lemma}\label{iwasawa-1} Let $w \in W$.
If $\gamma\in \Gamma\cap BwB$, then
\[
\mathrm{Iw}_{A^+}(\gamma g)=\mathrm{Iw}_{A^+}(g)^w\cdot\mathrm{Iw}_{A^+}(wu_w)
\]
for some $u_w\in U_w$ depending on $\gamma$ and $g$.
\end{lemma}

We also recall that  one has, for $u_w\in U_w$,
\begin{equation}\label{log}
\log \big(\mathrm{Iw}_{A^+}(wu_w)\big)=\sum_{\alpha\in \Phi_{w^{-1}}}c_\alpha \alpha^\vee\quad\textrm{with}\quad c_\alpha\leq 0
\end{equation}
by \cite[Lemma 6.1]{GMS2} or \cite[(3.25)]{GMP}.

To utilize Property RD, we need the property of the coefficient of $\alpha_P$ in the fundamental weight $\varpi_P$,  given by the following lemma.

\begin{lemma}\label{lem2-1}
The coefficient $n_{i_P}$ of $\alpha_P$ in the expansion $\varpi_P=\sum_{i\in I} n_i\alpha_i$ is negative.
\end{lemma}

\begin{proof}
By relabeling the simple roots, we may assume that $\alpha_P=\alpha_r$ and the upper-left $(r-1)\times (r-1)$ corner of $\mathsf{A}$ is the Cartan matrix of the root system $\Phi_M$ of $M$. Let $\mathsf{D}=\textrm{diag}(d_1,\ldots,d_r)$ be a symmetrizer of $\mathsf{A}$ and write
\[
\mathsf{DA}=\begin{pmatrix}
U & \beta \\
\beta^T & c
\end{pmatrix},\quad (\mathsf{DA})^{-1}=\begin{pmatrix} V & \gamma \\ \gamma^T & d\end{pmatrix},
\]
where $U, V$ are $(r-1)\times(r-1)$ matrices and $c, d$ are scalars.
Then $U=U^T$ is positive definite because $\Phi_M$ is of finite type. If $\gamma^T=(\gamma_1,\ldots,\gamma_{r-1})$, then
\[
\varpi_r=\sum^{r-1}_{i=1}d_r \gamma_i\alpha_i + d_r d\alpha_r.
\]
Hence we only need to show that $d<0$. 

The matrix $\mathsf{DA}$ has the same signature as
\[
\begin{pmatrix}
U & 0 \\
0 & c-\beta^TU^{-1}\beta
\end{pmatrix}.
\]
Since we assume that $\mathfrak{g}$ is of infinite type and $\mathsf{A}$ is nonsingular, we must have $c-\beta^TU^{-1}\beta<0$. We have the relations
\[
U\gamma+\beta d=0,\quad \beta^T\gamma+cd=1.
\]
It follows that
\[
\beta^T\gamma+\beta^TU^{-1}\beta d=\beta^TU^{-1}(U\gamma+\beta d)=0,
\]
which implies that $(c-\beta^TU^{-1}\beta)d=1$. Hence $d<0$ as desired.
\end{proof}

As a corollary, we obtain the following result.

\begin{lemma}\label{lem3-1}
For $D>0$ sufficiently small, $w^{-1}(D\varpi_P+\rho_M)$ is a positive linear combination of simple roots for any nontrivial $w\in W^\theta$.
\end{lemma}

\begin{proof}
Recall that $\rho_M$ is the half sum of the positive roots of $M$. Write $\varpi_P=\sum_{i\in I}n_i\alpha_i$ as in Lemma \ref{lem2-1}. Then we have 
\[
D\varpi_P+\rho_M=\sum_{i\in\theta} m_i\alpha_i +m_{i_P}\alpha_P,
\]
where 
\[
\sum_{i\in\theta}m_i\alpha_i=\rho_M+ D\sum_{i\in\theta}n_i\alpha_i,
\]
and $m_{i_P}=Dn_{i_P}<0$ by Lemma \ref{lem2-1}. Clearly for $D>0$ sufficiently small we have $m_i>0$ for $i\in \theta$. Consider a nontrivial element $w \in W^\theta$. Since $w^{-1}\alpha_i>0$ for $i\in \theta$, we must have
$w^{-1}\alpha_P<0$. The lemma follows immediately.
\end{proof}

With the above results, we can now state and prove the following inequality.

\begin{lemma}\label{ine} If $P$ satisfies Property RD, then for $D>0$ sufficiently small,
\[
\mathrm{Iw}_{A^+}(\gamma g)^{-\rho_M}\leq \mathrm{Iw}_{A^+}(\gamma g)^{D\varpi_P}
\]
for any $g\in UA_\mathcal{C}K$, $w\in W^\theta$ and $\gamma\in\Gamma\cap BwB$.
\end{lemma}

\begin{proof} From Definition \ref{ample}, there exists $D>0$ such that
\begin{equation} \label{DP}  \la D \varpi_P +\rho_M, \alpha^\vee \ra \le 0 \quad \text{for } \alpha \in \Phi_{w^{-1}}.\end{equation} By Remark \ref{rmk-sh}, we may shrink the constant  $D$,  if necessary, and  assume that the property in Lemma \ref{lem3-1} holds with this constant $D$. 
By Lemma \ref{iwasawa-1} and (\ref{log}), in order to prove Lemma~\ref{ine}, it suffices to establish its logarithmic form
\begin{equation}\label{log-iwa}
\langle w^{-1}(D\varpi_P+\rho_M),\,  \log(\textrm{Iw}_{A^+}(g))\rangle+\sum_{\alpha\in \Phi_{w^{-1}}}c_\alpha \langle D\varpi_P+\rho_M, \alpha^\vee\rangle\geq 0,
\end{equation}
with $c_\alpha\leq 0$, $\alpha\in \Phi_{w^{-1}}$. The first term is positive by Lemma \ref{lem3-1} and the assumption that $\log(\textrm{Iw}_{A^+}(g))\in \mathcal{C}$. The summation
in \eqref{log-iwa} is nonnegative by Property RD \eqref{DP}. Hence \eqref{log-iwa} holds.
\end{proof}

{\it Proof of Theorem \ref{hol}.} Fix any real number $s_0>\langle\rho_P,\alpha_P^\vee\rangle$, and assume that $\textrm{Re}(s)<s_0$.
By the rapid decay of cuspidal automorphic forms (e.g. \cite{MiSch}), for any $n>0$ there exists a constant $C_1>0$ depending on $n$ such that
\begin{equation}\label{decay}
|f(g)|\leq C_1 \, \textrm{Iw}_{A_1^+}(g)^{-n\rho_M}.
\end{equation}
We take 
\[
n:=\frac{s_0-\operatorname{Re}(s)}{D}>0,
\]
with $D$  the constant  in Lemma \ref{ine}.

Recall that we have the Bruhat decomposition (\cite{CG})
\[  G_\R = \bigcup_{w \in W^\theta} PwB.\]
 We claim that for any $\gamma\in (\Gamma\cap P)\backslash  ( \Gamma \cap PwB)$ with $w\in W^\theta$,  we can choose a representative of $\gamma$ which lies in $ \Gamma\cap BwB$. Recall the Levi decomposition $P=MN$. By the Iwasawa decomposition over $\mathbb{Q}$, we have a representative  of $\gamma$ of the form
 \[
 \tilde{\gamma}=mn wb\in \Gamma\cap PwB,
 \]
 where $m\in M_\mathbb{Q}$, $n\in N_\mathbb{Q}$ and $b \in B_\mathbb{Q}$. Since
$M$ is finite-dimensional, we have $M_\mathbb{Q}=\Gamma_M B_{M, \mathbb{Q}}$ (see \cite{Go}), where 
$\Gamma_M:=\Gamma\cap M$ and $B_{M,\mathbb{Q}}:=B_\mathbb{Q}\cap M$. Write 
$m=  \gamma_M b_M$ accordingly, where $\gamma_M\in \Gamma_M$ and $b_M\in B_{M, \mathbb{Q}}$. 
Then 
 \[
 \gamma_M^{-1}\tilde{\gamma}=b_Mnwb
 \]
 is a representative of $\gamma$, which clearly lies in $\Gamma\cap BwB$. 

With the claim established, it follows from (\ref{extra}), (\ref{levi}), Lemma \ref{ine} and (\ref{decay}) that
\begin{align*}
 \big| \textrm{Iw}_{H^+}(\gamma g)^{s\varpi_P} f\big(\textrm{Iw}_{M}(\gamma g)\big)\big|\ \
 \le \ \   &C_1  \, \textrm{Iw}_{H^+}(\gamma g)^{(\textrm{Re\,}s)\varpi_P}
 \textrm{Iw}_{A_1^+}\circ\textrm{Iw}_{L}(\gamma g)^{-n\rho_M}\\
 \leq  \ \ &   C_1 \, \textrm{Iw}_{H^+}(\gamma g)^{(\textrm{Re\,}s)\varpi_P}\textrm{Iw}_{H^+}(\gamma g)^{nD\varpi_P}\\
 = \ \ & C_1  \, \textrm{Iw}_{H^+}(\gamma g)^{s_0\varpi_P}.
\end{align*}
Note that the constants $C_1$ and $D$ are independent of $w$ and that the estimation works for any representative $\gamma \in \Gamma \cap BwB$. Multiplying by $\textrm{Iw}_{H^+}(\gamma g)^{\rho_P}$ and taking the summation over $\gamma$, it follows that $E_{f}(s,g)$ is bounded by
$C_1 E(s_0,g)$. Hence by Proposition \ref{cusp} the series $E_f(s,g)$ is absolutely convergent on $U\frak{S}K$ off $S_0K$, with $S_0$ a measure zero subset of $(\Gamma\cap U)\backslash U\mathfrak{S}$. Note that this  subset $S_0$ depends on $s_0$. To conclude, we may take a sequence of real numbers  $s_0^{(i)}>\langle\rho_P,\alpha_P^\vee\rangle$ which goes to infinity, and take the countable union of the corresponding measure zero subsets $S_0^{(i)}$.  \qed

Assuming Property \ref{conj}, we have everywhere convergence of cuspidal Eisenstein series over the full Tits cone by the reduction mechanism.   Thus in this case Proposition \ref{cusp} and Theorem \ref{hol} can be strengthened. We summarize this as the following result.

\begin{theorem} \label{hol-2} Assume Property \ref{conj}, and let $P=MN$ be a maximal parabolic subgroup of $G_\R$ with finite-dimensional Levi subgroup $M$. Then the following hold.

(i) For a bounded function $f$ on $M\cap \Gamma\backslash M /M\cap K$, the Eisenstein series $E_f(s,g)$ converges absolutely for $g\in \Gamma U A_\mathfrak{C} K$ and $\mathrm{Re}(s)>1$.

(ii) If $P$ satisfies Property RD and $f$ is an unramified cusp form on $L\cap\Gamma\backslash L$, then the Eisenstein series $E_f(s,g)$ is an entire function of $s\in\C$ for $g\in \Gamma U A_\mathcal{C} K$.
\end{theorem}

\section{Parabolic subgroups satisfying Property RD} \label{sec-ample}

In this section we investigate a sufficient condition which implies that every maximal parabolic subgroup of $G_\R$ with finite-dimensional Levi subgroup satisfies Property RD, and then discuss some examples, including  the rank 3 Feingold--Frenkel hyperbolic algebra $\mathcal{F}$.

We keep the notations of the previous section. In particular, $P=P_\theta$ stands for a maximal parabolic subgroup of $G_\R$, whose Levi subgroup $M$ is a finite-dimensional split reductive group. 

\begin{proposition}\label{lem4}
For any $w \in W$ and $\alpha \in \Phi_{w^{-1}}$, suppose that we have $\langle\alpha_i, \alpha^\vee\rangle\leq 0$ for any $\alpha_i$ such that $w^{-1}\alpha_i>0$. Then  every maximal parabolic subgroup $P$ with finite dimensional Levi subgroup satisfies Property RD.
\end{proposition}

\begin{proof}
As in the proof of Lemma \ref{lem3-1}, for small $D>0$ we have
\[
D\varpi_P+\rho_M=\sum_{i\in \theta} m_i\alpha_i +m_{i_P}\alpha_P
\]
where $m_i>0$ for $i\in \theta$ and $m_{i_P}<0$. Take $w\in W^\theta$, $w\neq\textrm{id}$, and $\alpha\in\Phi_{w^{-1}}$. Then by assumption
\begin{equation}\label{deltam}
\langle\alpha_i, \alpha^\vee\rangle\leq0\quad\textrm{for}\quad i\in\theta.
 \end{equation}
 Since $w_{\alpha_P}\alpha \in \Phi_{w^{-1}w_{\alpha_P}}$ and $w^{-1}w_{\alpha_P}\alpha_P=-w^{-1}\alpha_P>0$, we have by assumption
\begin{equation}\label{hp}
\langle\alpha_P, \alpha^\vee\rangle=-\langle\alpha_P, w_{\alpha_P}\alpha^\vee\rangle   \geq 0.
\end{equation}
From (\ref{deltam}) and (\ref{hp}) it follows that for any $\alpha\in\Phi_{w^{-1}}$ one has
\[
\langle D\varpi_P+\rho_M,\alpha^\vee\rangle=\sum_{i\in I} m_i \langle\alpha_i,\alpha^\vee\rangle\leq 0.
\]
This proves that $P$ satisfies Property RD.
\end{proof}

Now we present a family of groups with parabolic subgroups satisfying Property RD. 

\begin{proposition} \hfill
If $\mathfrak g$ is a rank 2 hyperbolic Kac--Moody algebra with generalized Cartan matrix $\mathsf{A}=\begin{pmatrix} ~2 & -b\\-a & ~2\end{pmatrix}$ where $a,b\geq 2$ and $ab\geq 5$ then every maximal parabolic subgroup of $G_\R$ satisfies Property RD.
\end{proposition}

\begin{proof}
First recall that $W$ is the infinite dihedral group generated by $w_1$ and $w_2$. Without loss of generality, assume that $w^{-1}=(w_2w_1)^m$ or $w_1(w_2w_1)^m$ for some $m \ge 0$. Then each $\alpha \in \Phi_{w^{-1}}$ has the form either $(w_1 w_2)^n \alpha_1$ or $(w_1 w_2)^n w_1 \alpha_2$ for some $n \ge 0$. As in the proof of \cite[Proposition 4.4]{CGLLM}, define 
\[ \mu = \frac {\sqrt{ab}+\sqrt{ab-4}} 2 \quad \text{ and } \quad h_n= \frac 1 {\mu - \mu^{-1}}(\mu^n - \mu^{-n}) . \] Then we get
\begin{align}
(w_1 w_2)^n \alpha_1 &= h_{2n+1} \alpha_1 + \tfrac {\sqrt a}{\sqrt b} h_{2n} \alpha_2 ,\label{w1w2-1}\\
(w_1 w_2)^n w_1 \alpha_2 &= \tfrac{\sqrt a}{\sqrt b} h_{2n+2} \alpha_1 + h_{2n+2} \alpha_2 .\label{w1w2-2} 
\end{align}

Since $w^{-1} \alpha_2 >0$ and $w^{-1} \alpha_1 <0$, we need only to show $\la \alpha_2, \alpha^\vee \ra <0$ for $\alpha \in \Phi_{w^{-1}}$ in order to apply Proposition \ref{lem4}. Because $\la \alpha_2, \alpha^\vee \ra$ and $\la \alpha, \alpha_2^\vee \ra$ have the same sign, we consider the latter for convenience.  
When $\alpha = (w_1 w_2)^n \alpha_1$, we obtain from \eqref{w1w2-1} and \cite[(4.10)]{CGLLM} 
\begin{align*} \la \alpha , \alpha_2^\vee \ra &= \la h_{2n+1} \alpha_1 + \tfrac {\sqrt a}{\sqrt b} h_{2n} \alpha_2, \alpha_2^\vee \ra \\
&= - \tfrac {\sqrt a}{\sqrt b} ( \sqrt{ab} h_{2n+1}- 2 h_{2n} ) 
\\& <0. \end{align*}
Similarly, when $\alpha = (w_1 w_2)^n w_1 \alpha_2$, we obtain from \eqref{w1w2-2} and \cite[(4.10)]{CGLLM} 
\begin{align*} \la \alpha , \alpha_2^\vee \ra &= \la \tfrac {\sqrt a}{\sqrt b} h_{2n+2} \alpha_1 +h_{2n+1} \alpha_2, \alpha_2^\vee \ra \\
&= - \sqrt{ab} h_{2n+2}+ 2 h_{2n+1}  \\
& <0. \end{align*}
Now it follows from Proposition \ref{lem4} that parabolic subgroups satisfy Property RD. 
\end{proof}

\begin{ex}[Non-example] \label{ex-non} 
Consider the following generalized Cartan matrix and its Dynkin diagram. \[\mathsf A =  {\scriptsize\begin{pmatrix} 2&-1&0&0&0&0&0 \\-1&2&-1&0&0&0&-1\\0&-1&2&-1&0&0&0\\0&0&-1&2&0&0& 0\\0&0&0&0&2&0&-1\\0&0&0&0&0&2&-1\\ 0&-1&0&0&-1&-1&2\end{pmatrix} } \qquad
\raisebox{2em}{\xymatrix@R=7ex{ *{\circ}<3pt> \ar@{-}[r]_<{1}  &
*{\circ}<3pt> \ar@{-}[r]_<{2}  &*{\circ}<3pt>
\ar@{-}[r]_<{3}  & *{\circ}<3pt> \ar@{}[r]_<{4} &\\  *{\circ}<3pt> \ar@{-}[r]^<{5}  &
*{\bullet}<3pt> \ar@{-}[r]^<{7}  \ar@{-}[u]  &  *{\circ}<3pt> \ar@{}[r]^<{6} &  }}
\]
Take $\theta=\{1,2,3,4,5,6\}$, $w = w_{\alpha_7}w_{\alpha_2} w_{\alpha_1}w_{\alpha_3}w_{\alpha_2}w_{\alpha_7}$ and $\alpha = \alpha_1+ 2\alpha_2+\alpha_3+\alpha_7$. Then $w \in W^\theta$ and $\alpha \in \Phi_{w^{-1}}$. 
Since $\langle \varpi_P, \alpha^\vee \rangle =1$ and $\langle \rho_M, \alpha^\vee \rangle=0$, we have \[  \langle D\varpi_P+\rho_M, \alpha^\vee \rangle = D .\] 
Thus the maximal parabolic subgroup $P_\theta$ does not satisfy Property RD.
\end{ex}

\begin{remark}
Generally, it seems difficult to show that a parabolic subgroup satisfies Property RD. However, some computer computations suggest that a large class of Kac--Moody groups have parabolic subgroups with Property RD.
In Example \ref{ex-non}, if we take $\theta=\{ 1,2,3,4,5,7\}$, then it appears that $P_\theta$ satisfies Property RD.
\end{remark}

In the rest of this section, we consider the rank 3 hyperbolic Kac--Moody algebra $\mathcal{F}$ studied by Feingold and Frenkel \cite{FF}, whose Cartan matrix is
\[
\mathsf{A}=(a_{ij})=\begin{pmatrix} 2 & -2 & 0 \\ -2 & 2 & -1 \\ 0 & -1 & 2\end{pmatrix}.
\]
 We will see below that
the condition of Proposition \ref{lem4}  fails for the algebra $\mathcal{F}$, but there are only finitely many exceptions and we can still show that every maximal parabolic subgroup with finite-dimensional Levi satisfies Property RD.

 Since  $\mathsf{A}$ is symmetric, we may identify $\langle\cdot,\cdot\rangle$ with the unique bilinear form $(\cdot|\cdot)$ on  $\mathfrak{h}^*$ satisfying $(\alpha_i|\alpha_j)=a_{ij}$, $1\le i,j \le 3$.
Define the linear isomorphism of vector spaces $\psi: \mathfrak{h}^* \overset{\cong}{\longrightarrow} S_2(\mathbb{C})$, the vector space of $2\times 2$ symmetric  complex matrices, by
\begin{equation}\label{psiFFdef}
\psi(x\alpha_1+y\alpha_2+z\alpha_3)=\left( \begin{matrix} y-z & y-x \\ y-x & z\end{matrix}\right),
\end{equation}
so that
\[
\psi(\alpha_1)=\begin{pmatrix} ~0 & -1 \\ -1 & ~0\end{pmatrix},\quad \psi(\alpha_2)=\begin{pmatrix} 1 & 1 \\ 1 & 0\end{pmatrix},
\quad \psi(\alpha_3)=\begin{pmatrix}-1 & 0 \\ ~0 & 1\end{pmatrix}
\]
and
$$(\lambda | \lambda) \ = \ -2\, \det\psi(\lambda),\ \lambda\in \frak{h}^*.$$
For convenience we shall henceforth identify $\frak{h}^*$ with $S_2(\C)$ using $\psi$. Then, for example, the inner product $(\cdot|\cdot)$ is given by
\[ \left ( \begin{pmatrix} \nu_1 & \nu_2 \\ \nu_2 & \nu_3 \end{pmatrix} \bigg| \begin{pmatrix} \mu_1 & \mu_2 \\ \mu_2 & \mu_3 \end{pmatrix} \right ) = - \nu_3 \mu_1 + 2  \nu_2 \mu_2 - \nu_1 \mu_3 .\]
The Weyl group $W$ of $\mathcal{F}$ is generated by
\begin{equation}\label{weyl}
w_1= \begin{pmatrix} 1 & 0 \\ 0 & -1 \end{pmatrix},\quad w_2= \begin{pmatrix} -1 & 1 \\ 0 & 1\end{pmatrix},\quad w_3= \begin{pmatrix}
0 & 1 \\ 1 & 0\end{pmatrix}
\end{equation}
and acts on  $S_2(\mathbb{C})$ by $g\cdot S=gSg^t$; it is
 isomorphic to $PGL_2(\mathbb{Z})$.

The set of  real roots  is given by
$$ \Phi_{\text{re}}=\left \{ \begin{pmatrix} n_1 & n_2 \\ n_2 & n_3 \end{pmatrix} \in S_2(\mathbb Z) \  \mid \  n_1n_3 -n_2^2=-1\right \},
$$
and
 the set of positive real roots is determined to be
\begin{equation}\label{real} \Phi_{\text{re}}^+=\left \{ \begin{pmatrix} n_1 & n_2 \\ n_2 & n_3 \end{pmatrix} \in S_2(\mathbb Z) \  \mid  \ \,
n_1n_3 -n_2^2=-1, \ n_1+n_3\geq n_2, \ n_1+n_3\geq 0,\  n_3\geq 0 \right \} .
\end{equation}
For the rest of this section we write
\begin{equation}\label{root}
\alpha=\begin{pmatrix} n_1 & n_2 \\ n_2 & n_3 \end{pmatrix}\in\Phi_{\textrm{re}}^+ \ \ \ \text{and} \ \ \  v=\begin{pmatrix} a & b \\ c & d\end{pmatrix}\in W,
\end{equation}
and frequently use the computation
\begin{equation}\label{action}
v\cdot\alpha =
\left( \begin{matrix}
a^2n_1+2abn_2+b^2n_3  & acn_1+(ad+bc)n_2+bdn_3\\
acn_1+(ad+bc)n_2+bdn_3 & c^2n_1+2cdn_2+d^2n_3
\end{matrix}
\right).
\end{equation}
In particular,
\begin{align}
v\cdot\alpha_1 \ = \ &
\left( \begin{matrix}
-2ab & -bc-ad\\
-bc-ad & -2cd
\end{matrix}
\right), \label{va1}\\
v\cdot\alpha_2 \ = \ &
\left( \begin{matrix}
a^2+2ab & ac+bc+ad\\
 ac+bc+ad & c^2+2cd
\end{matrix}
\right),\label{va2}\\
\text{and} \ \ \
v\cdot\alpha_3 \ = \ &
\left( \begin{matrix}
b^2-a^2 & bd-ac\\
bd-ac & d^2-c^2
\end{matrix}
\right).\label{va3}
\end{align}

Since the elements of $\Phi_v$ are precisely the positive real roots flipped by $v$, a root $\alpha\in\Phi_v$ is constrained by the inequalities
\begin{align}
& c^2n_1+2cdn_2+d^2n_3\leq 0,\label{neq1}\\
& (a^2+c^2)n_1+2(ab+cd)n_2+(b^2+d^2)n_3\leq 0, \ \   \text{and} \label{neq2}\\
& (a^2+c^2)n_1+2(ab+cd)n_2+(b^2+d^2)n_3\leq acn_1+(ad+bc)n_2+bdn_3,  \label{neq3}
\end{align}
in addition to  (\ref{real}).

We have two standard maximal parabolic subgroups $P_1$ and $P_2$ with finite dimensional Levi subgroups $M_1$ and $M_2$, which correspond to simple roots $\alpha_1$ and $\alpha_2$ respectively. Note that in terms of previous notations, they have semisimple subgroups $L_1\cong SL(3)$ and $L_2\cong SL(2)\times SL(2)$ respectively.

We will use the following lemma  to prove that the parabolic subgroups $P_1$ and $P_2$ satisfy Property RD (Proposition \ref{FF-1}). The proof of this lemma is by direct calculation and will be omitted. 

\begin{lemma}\label{fflem}
If $w\in W$, $w\neq\mathrm{id}$ and $w^{-1}\alpha_i>0$, then $\langle\alpha_i,\alpha^\vee\rangle\leq 0$ for any $\alpha\in\Phi_{w^{-1}}$ unless

(i) $i=2$ and $\alpha$ equals
\[
\beta_1=\begin{pmatrix} 0 & 1 \\ 1 & 1\end{pmatrix}\quad\textrm{or}\quad \beta_2=\begin{pmatrix} 1 & 2 \\ 2 & 3\end{pmatrix},
\]
or
(ii) $i=3$ and $\alpha$ equals $\beta_1$ as above or
\[
\beta_3=\begin{pmatrix} 0 & -1 \\ -1 & 1\end{pmatrix}.
\]
Moreover in these exceptional cases one has $\langle\alpha_i,\alpha^\vee\rangle=1$.
\end{lemma}

\begin{proposition}\label{FF-1}
The parabolic subgroups $P_1$ and $P_2$ both satisfy Property RD.
\end{proposition}

\begin{proof}     Recall that for $j=1,2$ and $D>0$ sufficiently small, one has
\[
D\varpi_j+\rho_{M_j}=\sum_i m_i\alpha_i
\]
where $m_j<0$ and $m_i>0$ for $i\in \theta_j:=\{1,2,3\}\setminus\{j\}$. Then we need to show that taking smaller $D$ if necessary,  one has
\[
\langle D\varpi_j+\rho_{M_j},\alpha^\vee\rangle=\sum_i m_i \langle\alpha_i, \alpha^\vee\rangle<0
\]
for any $w\in W^{\theta_j}$ and $\alpha\in\Phi_{w^{-1}}$. Note that $\langle\alpha_i, \alpha^\vee\rangle$, $i=1,2,3$, cannot be all zero. We shall consider $P_1$ and $P_2$ separately.

(1) Assume that $j=1$. Then $w^{-1}\alpha_1<0$ and $w^{-1}\alpha_i>0$ for $i=2,3$. By Lemma \ref{fflem}  and (\ref{hp}), one always has
$\langle\alpha_1,\alpha^\vee\rangle\geq 0$ for $\alpha\in\Phi_{w^{-1}}$, and it suffices to consider the exceptional cases when $\alpha=\beta_i$ for some $i=1,2,3$.

Since $\beta_1=\alpha_2+\alpha_3$,
we have $w^{-1}\beta_1>0$, hence $\alpha=\beta_1$ cannot happen. Since $\langle\alpha_3, \beta_2^\vee\rangle=2$, it follows from Lemma \ref{fflem} that $\alpha=\beta_2$ cannot happen either. Finally if $\alpha=\beta_3$, then $\langle\rho_{M_1}, \beta^\vee_3\rangle=-2<0$, hence $\langle D\varpi_1+\rho_{M_1}, \beta_3^\vee\rangle=D-2<0$ for $D<2$.

(2) Assume that $j=2$. Then $w^{-1}\alpha_2<0$ and $w^{-1}\alpha_i>0$ for $i=1,3$.  By Lemma \ref{fflem} and (\ref{hp}), one has $\langle\alpha_2,\alpha^\vee\rangle\geq 0$ for $\alpha\in\Phi_{w^{-1}}$ unless $\alpha=w_2\beta_1$ or $w_2\beta_2$. However,  we find that $w_2\beta_1=\alpha_3$ and $\langle\alpha_3, w_2\beta_2^\vee\rangle=3$, hence $\alpha\neq w_2\beta_1, w_2\beta_2$, thanks to $w^{-1}\alpha<0$ and Lemma \ref{fflem} respectively. Thus we have proved that $\langle \alpha_2,\alpha^\vee\rangle\geq 0$ always holds.

By Lemma \ref{fflem} again, it suffices to consider the exceptional cases $\alpha=\beta_1$ or $\beta_3$. For $\alpha=\beta_1$ we have
$\langle\rho_{M_2}, \beta^\vee_1\rangle=-1/2<0$, hence $\langle D\varpi_2+\rho_{M_2}, \beta_1^\vee\rangle=D-1/2<0$ for $D<1/2$. The case $\alpha=\beta_3$ cannot happen, because $\langle\alpha_1,\beta_3^\vee\rangle=2$ which violates Lemma \ref{fflem}.
\end{proof}

\end{document}